\documentclass[11pt,reqno]{amsart}

\usepackage[T1]{fontenc}
\usepackage{lmodern}
\usepackage{amsmath,amssymb,amsthm,mathtools}
\usepackage{tikz-cd}
\usepackage{enumitem}
\usepackage{microtype}
\usepackage[margin=1.1in]{geometry}
\usepackage[hidelinks]{hyperref}

\newtheorem{theorem}{Theorem}[section]
\newtheorem{proposition}[theorem]{Proposition}
\newtheorem{lemma}[theorem]{Lemma}
\newtheorem{corollary}[theorem]{Corollary}
\theoremstyle{definition}
\newtheorem{remark}[theorem]{Remark}
\newtheorem{example}[theorem]{Example}
\numberwithin{equation}{section}

\newcommand{\Z}{\mathbb{Z}}
\newcommand{\Aff}{\mathbb{A}}
\newcommand{\PP}{\mathbb{P}}
\newcommand{\eff}{\mathrm{eff}}
\newcommand{\red}{\mathrm{red}}
\newcommand{\DM}{\mathbf{DM}}
\newcommand{\MDM}{\mathbf{MDM}}
\newcommand{\uMDM}{\mathbf{\underline{M}DM}}
\newcommand{\MCor}{\mathbf{MCor}}
\newcommand{\uMCor}{\mathbf{\underline{M}Cor}}
\newcommand{\Cor}{\mathbf{Cor}}
\newcommand{\uM}{\underline{M}}
\newcommand{\MV}{M^{\mathrm V}}
\newcommand{\one}{\mathbf{1}}
\newcommand{\cbar}{\overline{\square}}
\newcommand{\id}{\mathrm{id}}
\DeclareMathOperator{\Spec}{Spec}

\DeclareMathOperator{\Cone}{Cone}

\title[Motives with modulus with finite coefficients]
{Voevodsky motives and motives with modulus\\
with finite coefficients in characteristic zero}
\author{Keiho Matsumoto}
\address{Graduate School of Science, The University of Osaka, Japan}
\email{lightkun0526@gmail.com}
\thanks{The author is supported by JSPS KAKENHI Grant Number JP25KJ0210.}
\date{}

\begin{document}

\begin{abstract}
Let $k$ be a field of characteristic zero and let $n\geq 2$ be an
integer. We prove that, with $\Z/n\Z$-coefficients, the category of
effective motives with modulus is equivalent to the category of
effective Voevodsky motives.
\end{abstract}

\maketitle

\section{Introduction}\label{intro}

Let $k$ be a field of characteristic zero. Kahn--Miyazaki--Saito--Yamazaki
introduced motives with modulus in \cite{KMSY21a,KMSY21b,KMSY22}.
Their category $\MDM^\eff(k)$ is generated by proper modulus pairs
$(X,D)$, where $D$ is an effective Cartier divisor and
$X\setminus |D|$ is smooth over $k$. Taking the interior induces a
symmetric monoidal localization
\[
 \omega_{\eff}\colon
 \MDM^\eff(k)\longrightarrow\DM^\eff(k)
\]
to Voevodsky's category of effective motives; see
\cite{V00} and \cite[Proposition~6.1.2]{KMSY22}.
We prove that this functor becomes an equivalence with finite
coefficients.

\begin{theorem}\label{main}
Let $k$ be a field of characteristic zero and let $n\geq 2$ be an
integer. Put $\Lambda=\Z/n\Z$. Then the canonical functor
\[
 \omega_{\eff}\colon
 \MDM^\eff(k,\Lambda)\longrightarrow\DM^\eff(k,\Lambda)
\]
is an equivalence of symmetric monoidal triangulated categories.
\end{theorem}

This is a categorical analogue of Miyazaki's independence theorem
for higher Chow groups with modulus
\cite[Theorem~1.3 and Corollary~5.2]{M19}. In characteristic zero,
that theorem says that torsion coefficients remove the dependence
on the multiplicities of the modulus. In particular, for a smooth
equidimensional $k$-scheme $X$ and an effective Cartier divisor $D$,
the natural map
\[
 \mathrm{CH}^r(X\mid D,q;\Lambda)
 \xrightarrow{\sim}
 \mathrm{CH}^r(X\mid D_\red,q;\Lambda)
\]
is an isomorphism for all $r,q\geq 0$.
There is also a sheaf-theoretic analogue: by
\cite[Theorem~3.5(1)]{BCKS17}, a torsion reciprocity presheaf with
transfers in characteristic zero is homotopy invariant if it is
separated for the Zariski topology. Thus reciprocity Nisnevich
sheaves with transfers and homotopy invariant Nisnevich sheaves
with transfers coincide with $\Lambda$-coefficients.

In positive characteristic, a comparison with $\Z[1/p]$-coefficients
is known under log resolution of singularities
\cite[Theorem~8.9]{M23a}. More precisely, for a perfect field $k_0$
of characteristic $p>0$ satisfying this resolution hypothesis,
\[
 \omega_{\eff}\colon
 \MDM^\eff(k_0,\Z[1/p])
 \xrightarrow{\sim}
 \DM^\eff(k_0,\Z[1/p])
\]
is an equivalence.

The essential calculation in the present proof concerns the pairs
$B_m=(\PP^1,m\{\infty\})$. Let $e_m$ be the constant correspondence
on $B_m$ with value $0$. We construct an integral correspondence
$\beta_{m,n}\colon B_m\rightsquigarrow B_m$ satisfying
\begin{equation}\label{introidentity}
 n\beta_{m,n}
 =\id+\bigl(n(m-1)-1\bigr)e_m
 \quad\text{in }\MDM^\eff(k,\Z).
\end{equation}
The construction uses a truncated $n$-th root of $1-T$.
The truncation makes the low-order terms cancel in the homotopy
polynomial; this cancellation gives exactly the pole bound required
by the modulus condition. Modulo $n$, \eqref{introidentity} contracts
$B_m$ to the origin. Nisnevich excision then removes the
multiplicities of an arbitrary strict normal crossings modulus.
The rest follows from log resolution and the reduced-boundary
comparison \cite[Theorem~7.1]{M23a}.

\subsection*{Notation}
Except in the positive-characteristic comparison above, $k$ has
characteristic zero and $\Lambda=\Z/n\Z$ with $n\geq 2$.
All schemes are separated and of finite type over the base field.
We allow the empty divisor. A strict normal crossings divisor has
smooth irreducible components meeting transversally. We write $|D|$
for the support of an effective Cartier divisor $D$ and, when the
ambient scheme is smooth, $D_\red$ for its reduced divisor.
All tensor products and coefficient changes are derived.
Mapping cones of the explicit complexes used below are formed
before passing to the Verdier quotient.

\section{Preliminaries}\label{prelim}

\subsection{Modulus pairs and correspondences}

We recall the definitions of modulus pairs and admissible
correspondences from \cite[Section~1]{KMSY21a}.
A modulus pair is a pair $\mathcal X=(X,D)$ consisting of a scheme
$X$ and an effective Cartier divisor $D$ such that
$\mathcal X^\circ=X\setminus |D|$ is smooth over $k$.
It is proper if $X$ is proper.

Let $\mathcal X=(X,D)$ and $\mathcal Y=(Y,E)$ be modulus pairs,
and let $V\subset\mathcal X^\circ\times\mathcal Y^\circ$ be an
integral finite correspondence. Write $\overline V$ for its closure
in $X\times Y$ and $\widetilde V$ for the normalization of
$\overline V$. With the induced projections denoted by $p_X,p_Y$,
the admissibility condition is
\begin{equation}\label{admissibility}
 p_X^*D\geq p_Y^*E\quad\text{on }\widetilde V.
\end{equation}
The correspondence is left proper if $\overline V\to X$ is proper.
The category $\uMCor(k)$ has modulus pairs as objects and integral
linear combinations of admissible left proper finite correspondences
as morphisms. Its full subcategory on proper pairs is $\MCor(k)$.
We denote the ordinary category of finite correspondences between
smooth $k$-schemes by $\Cor(k)$.

The tensor product of modulus pairs is
\[
 (X,D)\otimes(Y,E)
 =\bigl(X\times Y,\operatorname{pr}_X^*D+
                         \operatorname{pr}_Y^*E\bigr).
\]
Put $B_m=(\PP^1,m\{\infty\})$ for $m\geq 1$ and
$\cbar=B_1$.

\subsection{The categories of motives}\label{categories}

For $R=\Z$ or $\Lambda$, we define $\MDM^\eff(k,R)$ and
$\uMDM^\eff(k,R)$ as in
\cite[Definitions~3.1.1 and 3.2.4]{KMSY22}, using $R$-valued
presheaves with transfers. Namely, we take the unbounded derived
categories of these presheaves on $\MCor(k)$ and $\uMCor(k)$,
respectively, and form the Verdier quotients by the localizing
subcategories generated by cube-invariance and Mayer--Vietoris
relations. We write $\DM^\eff(k,R)$ for Voevodsky's category of
effective motives \cite{V00}, with $R$-coefficients, using its
unbounded version as in \cite{KMSY22}.

The images of the representable presheaves are denoted by
$M(X,D)_R$ for a proper pair and $\uM(X,D)_R$ for an arbitrary
pair. For a smooth scheme $U$, its Voevodsky motive is denoted by
$\MV(U)_R$. We write $\one_R$, or simply $R$, for the tensor unit.

The same arguments as in the integral case
\cite[Theorems~3.3.1 and 5.2.2, and
Propositions~6.1.1--6.1.2]{KMSY22} give symmetric monoidal functors
\[
\begin{aligned}
 \tau_{\eff}&\colon
 \MDM^\eff(k,R)\longrightarrow\uMDM^\eff(k,R),\\
 \omega_{\eff}&\colon
 \MDM^\eff(k,R)\longrightarrow\DM^\eff(k,R),\\
 \underline\omega_{\eff}&\colon
 \uMDM^\eff(k,R)\longrightarrow\DM^\eff(k,R),
\end{aligned}
\]
with right adjoints $\tau^\eff$, $\omega^\eff$, and
$\underline\omega^\eff$, respectively. The functors
$\tau_{\eff}$, $\omega^\eff$, and $\underline\omega^\eff$ are fully
faithful, and all six functors preserve coproducts. Moreover,
$\omega_{\eff}=\underline\omega_{\eff}\tau_{\eff}$ and
$\omega^\eff\simeq\tau^\eff\underline\omega^\eff$.
The motives of proper pairs, arbitrary pairs, and smooth schemes
compactly generate the respective categories.

For an integral motive $A$, write $A_\Lambda$ for its derived
extension of coefficients. The comparison adjunctions are compatible
with this extension, including their units and counits. In particular,
$(\omega^\eff_\Z B)_\Lambda\simeq\omega^\eff_\Lambda(B_\Lambda)$ for
$B\in\DM^\eff(k,\Z)$.

\section{The motive of a thickened interval}\label{interval}

For $m\geq 1$, let $p_m\colon B_m\to\Spec k$ be the structure
morphism, and let $i_m\colon\Spec k\to B_m$ be the section with value
$0$. Put $e_m=i_mp_m$. We use the same notation for the induced
morphisms of motives.

\begin{proposition}\label{identity}
Let $m\geq 2$ and $n\geq 2$. There exists an admissible
correspondence with integral coefficients
$\beta_{m,n}\colon B_m\rightsquigarrow B_m$ such that
\begin{equation}\label{integralidentity}
 n\beta_{m,n}
 =\id_{M(B_m)_\Z}+\bigl(n(m-1)-1\bigr)e_m
\end{equation}
in $\MDM^\eff(k,\Z)$.
\end{proposition}

\begin{proof}
Put $r=m-1$ and $N=nr$. We construct an admissible correspondence
\[
 \Gamma\colon B_m\otimes\cbar\rightsquigarrow B_m
\]
whose restrictions to the two endpoints are the two sides of
\eqref{integralidentity}.

\medskip
\noindent\textit{Step 1: Construction of $\beta_{m,n}$.}
Since $k$ has characteristic zero, the polynomial
\[
 g(T)=\sum_{j=0}^{r}c_jT^j,
 \qquad
 c_j=(-1)^j\binom{1/n}{j},
\]
is defined over $k$. It is the truncation of $(1-T)^{1/n}$ modulo
$T^m$. Consequently,
\begin{equation}\label{rootcongruence}
 g(T)^n\equiv 1-T\pmod{T^m}.
\end{equation}
Let
\[
 P(x,y)=y^rg(x/y)=\sum_{j=0}^{r}c_jx^jy^{r-j}.
\]
This polynomial is homogeneous of degree $r$ and monic in $y$.
Therefore $V(P)\subset\Aff^1_x\times\Aff^1_y$ is finite flat of degree
$r$ over $\Aff^1_x$. Its fundamental cycle defines an integral-coefficient
finite correspondence, which we denote by $\beta_{m,n}$.

We check its admissibility. Let $Z$ be an irreducible component of its
support, let $\widetilde Z$ be the normalization of its closure in
$\PP^1_x\times\PP^1_y$, and let $v$ be the discrete valuation at a
codimension-one point of $\widetilde Z$. Put
\[
 a_x=\max\{0,-v(x)\},\qquad a_y=\max\{0,-v(y)\}.
\]
The coefficients of the pullbacks of $\{\infty_x\}$ and
$\{\infty_y\}$ at this point are $a_x$ and $a_y$, respectively.
We claim that $a_y\leq a_x$. Otherwise $a_y>a_x$, and the term $y^r$ in
$P(x,y)$ has valuation $-ra_y$, whereas, for $j\geq 1$, every nonzero
term satisfies
\[
 v(c_jx^jy^{r-j})
 \geq-ja_x-(r-j)a_y
 =-ra_y+j(a_y-a_x)>-ra_y.
\]
This contradicts $P(x,y)=0$, since a vanishing sum cannot have a unique
term of smallest valuation. Thus $ma_y\leq ma_x$, proving admissibility.
Left properness follows from the properness of the target $\PP^1$.

\medskip
\noindent\textit{Step 2: Construction of the homotopy.}
Define
\begin{equation}\label{homotopypolynomial}
 H(x,t,y)=(1-t)P(x,y)^n+t\,y^{N-1}(y-x).
\end{equation}
The leading coefficient in $y$ is $(1-t)+t=1$. Hence $H$ is monic of
degree $N$ in $y$, and $V(H)$ is finite flat over
$\Aff^1_x\times\Aff^1_t$. Let $\Gamma$ denote its fundamental cycle.

By \eqref{rootcongruence}, there are coefficients $d_j\in k$ such that
\begin{equation}\label{hexpansion}
 H(x,t,y)=y^N-xy^{N-1}
             +\sum_{j=m}^{N}(1-t)d_jx^jy^{N-j}.
\end{equation}
In particular, the terms $x^jy^{N-j}$ with $2\leq j\leq m-1$ do not
occur. Here $N\geq m$.

\medskip
\noindent\textit{Step 3: Admissibility of the homotopy.}
Let $Z$ be an irreducible component of the support of $\Gamma$.
Normalize its closure in
$\PP^1_x\times\PP^1_t\times\PP^1_y$, and let $v$ be the valuation at
an arbitrary codimension-one point of this normalization. Put
\[
 a_x=\max\{0,-v(x)\},\qquad
 a_t=\max\{0,-v(t)\},\qquad
 a_y=\max\{0,-v(y)\}.
\]
The required modulus inequality is
\begin{equation}\label{polebound}
 ma_y\leq ma_x+a_t.
\end{equation}
Indeed, the source modulus is
$m\{\infty_x\}+\{\infty_t\}$ and the target modulus is
$m\{\infty_y\}$.

Suppose that \eqref{polebound} fails. Then
$a_y>a_x+a_t/m$, so $a_y>0$ and
\[
 v(y^N)=-Na_y.
\]
The second term in \eqref{hexpansion} satisfies
\[
 v(xy^{N-1})\geq-a_x-(N-1)a_y
                 =-Na_y+(a_y-a_x)>-Na_y.
\]
Since $v(1-t)\geq-a_t$, for $j\geq m$ we also have
\begin{align*}
 v\bigl((1-t)d_jx^jy^{N-j}\bigr)
 &\geq-a_t-ja_x-(N-j)a_y\\
 &=-Na_y+j(a_y-a_x)-a_t\\
 &>-Na_y
\end{align*}
whenever the corresponding term is nonzero. The last inequality uses
$j\geq m$ and $a_y-a_x>a_t/m$. Again $y^N$ would be the unique term of
smallest valuation in a vanishing sum. This contradiction proves
\eqref{polebound}.

Divisor inequalities on a normal noetherian scheme can be checked at
codimension-one points. We have therefore proved admissibility.
The closure of $Z$ is proper over $\PP^1_x\times\PP^1_t$, because
$\PP^1_y$ is proper. Thus $\Gamma$ is a morphism in $\MCor(k)$ from
$B_m\otimes\cbar$ to $B_m$.

\medskip
\noindent\textit{Step 4: Restriction to the endpoints.}
At $t=0$ and $t=1$, respectively, \eqref{homotopypolynomial} becomes
\[
 H(x,0,y)=P(x,y)^n,
 \qquad
 H(x,1,y)=y^{N-1}(y-x).
\]
As $V(H)$ is finite flat over the source, composition with an endpoint
section is given by the corresponding fiber cycle, with its
scheme-theoretic multiplicities. We obtain
\[
 \Gamma|_{t=0}=n\beta_{m,n},
 \qquad
 \Gamma|_{t=1}=\Delta_{\Aff^1}+(N-1)(\Aff^1\times\{0\}).
\]
The second summand in the last expression represents $(N-1)e_m$.

Let $j_0,j_1\colon B_m\to B_m\otimes\cbar$ be the two endpoint
sections, and let $\pi\colon B_m\otimes\cbar\to B_m$ be the projection.
The morphism $M(\pi)$ is an isomorphism by cube invariance, and
$\pi j_0=\pi j_1=\id$. Hence $M(j_0)=M(j_1)$. Composing this equality with
$M(\Gamma)$ proves \eqref{integralidentity}.
\end{proof}

\begin{remark}\label{denominators}
The coefficients $\binom{1/n}{j}$ belong to the base field $k$.
They are coefficients of equations defining finite correspondences,
not coefficients of cycles. Both $\beta_{m,n}$ and $\Gamma$ are
cycles with integer multiplicities. Thus no inversion of $n$ in
the coefficient ring of motives occurs in Proposition~\ref{identity}.
\end{remark}

\begin{corollary}\label{contract}
For every $m\geq 1$, the structure morphism is an isomorphism
\[
 p_m\colon M(B_m)_\Lambda\xrightarrow{\ \sim\ }\Lambda,
\]
with inverse $i_m$. In particular, for $a\geq b\geq 1$, the morphism
\[
 M(B_a)_\Lambda\longrightarrow M(B_b)_\Lambda
\]
induced by the identity on $\Aff^1$ is an isomorphism.
\end{corollary}

\begin{proof}
For $m=1$, the first assertion is cube invariance. For $m\geq 2$,
apply extension of coefficients to \eqref{integralidentity}.
Since $n=0$ in $\Lambda$ and $n(m-1)-1=-1$ in $\Lambda$, it gives
\[
 0=\id_{M(B_m)_\Lambda}-e_m.
\]
Thus $i_mp_m=\id$, while $p_mi_m=\id$ holds before localization.
This proves the first assertion. The morphism in the second assertion
commutes with the two structure morphisms, both of which are
isomorphisms.
\end{proof}

\begin{example}
For $m=n=2$, the construction reads
\[
 P(x,y)=y-\frac{x}{2},\qquad
 H(x,t,y)=y^2-xy+\frac{1-t}{4}x^2.
\]
It gives the integral identity
\[
 2[\Gamma_{x\mapsto x/2}]=[\Delta_{\Aff^1}]+e_2
\]
on $M(\PP^1,2\{\infty\})_\Z$.
\end{example}

\section{Removing the multiplicities of the modulus}\label{multiplicity}

\subsection{Excision}

Let $X$ be smooth, let $Z\subset X$ be a smooth integral effective
Cartier divisor, and let $E$ be an effective Cartier divisor not
containing $Z$. Suppose that $|E|\cup Z$ is strict normal crossings.
For $a\geq b\geq 1$, set
\begin{equation}\label{defcone}
 C_{a,b}(X;E,Z)
 =\Cone\bigl(\uM(X,E+aZ)_\Lambda
                   \longrightarrow\uM(X,E+bZ)_\Lambda\bigr).
\end{equation}
This is the image of the corresponding two-term complex of
representable presheaves in degrees $-1,0$.
By \cite[Theorem~3.2]{M23a}, after coefficient extension, an
\'etale morphism $f\colon X'\to X$ inducing an isomorphism
$Z'=f^{-1}Z\to Z$ gives an isomorphism
\begin{equation}\label{excision}
 C_{a,b}(X';f^*E,Z')\xrightarrow{\sim}C_{a,b}(X;E,Z).
\end{equation}
The cokernel formulation in that theorem agrees with
\eqref{defcone}, since the maps of integral representables induced
by decreasing the modulus are injective.
We also use this notation when $Z$ is empty, in which case the
cone is zero.

\begin{lemma}\label{affineline}
For $a\geq b\geq 1$, the natural morphism
\[
 \uM(\Aff^1,a\{0\})_\Lambda
 \longrightarrow\uM(\Aff^1,b\{0\})_\Lambda
\]
is an isomorphism.
\end{lemma}

\begin{proof}
Excision for $\Aff^1\hookrightarrow\PP^1$ identifies its cone with
\[
 \Cone\bigl(\uM(\PP^1,a\{0\})_\Lambda
                    \longrightarrow\uM(\PP^1,b\{0\})_\Lambda\bigr).
\]
Exchanging $0$ and $\infty$ identifies this with the cone of
$\uM(B_a)_\Lambda\to\uM(B_b)_\Lambda$, which vanishes by
Corollary~\ref{contract}.
\end{proof}

\subsection{Strict normal crossings divisors}

\begin{proposition}\label{snc}
Let $X$ be smooth and let $D\geq D'$ be effective Cartier divisors
with the same strict normal crossings support. Then
\[
 \uM(X,D)_\Lambda\longrightarrow\uM(X,D')_\Lambda
\]
is an isomorphism. In particular,
$\uM(X,D)_\Lambda\simeq\uM(X,D_\red)_\Lambda$.
If $X$ is proper, the same assertions hold in
$\MDM^\eff(k,\Lambda)$.
\end{proposition}

\begin{proof}
We change one multiplicity at a time. Write the two divisors as
$E+aZ$ and $E+bZ$, where $a\geq b\geq 1$ and $Z$ is a smooth
irreducible component of their common support. It suffices to
show that $C_{a,b}(X;E,Z)=0$.

The immersion $(Z,E|_Z)\hookrightarrow(X,E)$ is transversal in
the sense of \cite[Definition~7]{KS21}. By
\cite[Lemma~8]{KS21}, there is a finite Zariski cover
$X=U_0\cup\cdots\cup U_s$, with $U_0=X\setminus Z$, and, for
$i\geq 1$, common \'etale neighborhoods
\begin{equation}\label{commonneighbor}
 U_i\xleftarrow{\pi_1}V_i\xrightarrow{\pi_2}Z_i\times\Aff^1,
 \qquad Z_i=Z\cap U_i.
\end{equation}
Put $E_i=E|_{U_i}$ and $E_{Z_i}=E|_{Z_i}$.
Both morphisms in \eqref{commonneighbor} are isomorphisms over the
distinguished copy of $Z_i$, and their pullback divisors satisfy
\[
 \pi_1^{-1}Z_i=\pi_2^{-1}(Z_i\times\{0\})\simeq Z_i,
 \qquad
 \pi_1^*E_i=\pi_2^*\operatorname{pr}_{Z_i}^*E_{Z_i}.
\]
Thus the closed immersions fit into the diagram
\[
\begin{tikzcd}[column sep=large]
 Z_i\arrow[d,hook]&
 Z_i\arrow[l,equal]\arrow[r,equal]\arrow[d,hook]&
 Z_i\arrow[d,hook,"0"]\\
 U_i&V_i\arrow[l,"\pi_1"']\arrow[r,"\pi_2"]&Z_i\times\Aff^1.
\end{tikzcd}
\]
Excision on both sides of \eqref{commonneighbor} gives
\begin{align*}
 C_{a,b}(U_i;E_i,Z_i)
 &\simeq
 C_{a,b}(Z_i\times\Aff^1;
              \operatorname{pr}_{Z_i}^*E_{Z_i},Z_i\times\{0\})\\
 &\simeq
 \uM(Z_i,E_{Z_i})_\Lambda\otimes
 C_{a,b}(\Aff^1;0,\{0\})
 =0.
\end{align*}
The second isomorphism follows from the tensor product of modulus
pairs and exactness of tensoring; the last equality is
Lemma~\ref{affineline}.

The same argument applies on every open subset $U\subset U_i$.
Indeed, replace $V_i$ by
\[
 \pi_1^{-1}U\cap
 \pi_2^{-1}\bigl((Z_i\cap U)\times\Aff^1\bigr).
\]
This is again a common \'etale neighborhood with the same pullback
identities. Hence the cones vanish on all finite intersections of
the cover. They also vanish on $U_0$.
Applying Mayer--Vietoris to the two-term complexes defining
\eqref{defcone} yields $C_{a,b}(X;E,Z)=0$.

Successively reducing the multiplicities proves the assertion in
$\uMDM^\eff$. For proper $X$, it follows in $\MDM^\eff$ from the
full faithfulness of $\tau_{\eff}$.
\end{proof}

\section{Comparison with Voevodsky motives}\label{comparison}

\subsection{Reduced boundaries}

\begin{proposition}\label{reducedcomparison}
Let $X$ be smooth and proper, let $E$ be a reduced strict normal
crossings divisor, and put $U=X\setminus E$. For $R=\Z$ or
$\Lambda$, the adjunction unit
\[
 M(X,E)_R\longrightarrow\omega^\eff\MV(U)_R
\]
is an isomorphism.
\end{proposition}

\begin{proof}
For integral coefficients, \cite[Theorem~7.1]{M23a} gives the unit
isomorphism in $\uMDM^\eff$. Apply $\tau^\eff$ and use
$\tau^\eff\tau_{\eff}\simeq\id$ and
$\tau^\eff\underline\omega^\eff\simeq\omega^\eff$.
For the empty boundary, use \cite[Theorem~6.3.1]{KMSY22}.
The finite-coefficient assertion follows by derived extension of
coefficients and compatibility of the comparison adjunctions with
this extension.
\end{proof}

\subsection{Proof of the main theorem}

\begin{proof}[Proof of Theorem~\ref{main}]
As in \cite[proof of Theorem~8.9]{M23a}, it suffices to check the
unit on smooth proper pairs with strict normal crossings support.
Indeed, log resolution \cite{T18} replaces every proper modulus
pair by such a pair, without changing its object in $\MCor(k)$;
see \cite[Theorem~1.2.3]{KMSY22} or \cite[Remark~2.3]{M23a}.
Their motives compactly generate, the right adjoint is fully
faithful, and both adjoint functors preserve coproducts.

For a smooth proper pair $(X,D)$ with strict normal crossings
support, Propositions~\ref{snc} and \ref{reducedcomparison} give
\[
 M(X,D)_\Lambda
 \xrightarrow{\sim}M(X,D_\red)_\Lambda
 \xrightarrow{\sim}\omega^\eff\MV(X\setminus|D|)_\Lambda.
\]
The composite is the unit by naturality. This proves the
equivalence. The functor is symmetric monoidal and preserves
compact objects, as does its inverse.
\end{proof}

\begin{corollary}\label{arbitrarymodulus}
For any proper modulus pair $(X,D)$, there is a natural isomorphism
\[
 M(X,D)_\Lambda\xrightarrow{\sim}
 \omega^\eff\MV(X\setminus|D|)_\Lambda.
\]
If $D\geq D'$ are effective Cartier divisors on $X$ with the same
support, then $M(X,D)_\Lambda\to M(X,D')_\Lambda$ is an isomorphism.
\end{corollary}

\begin{proof}
The first assertion is the unit isomorphism. For the second,
$\omega_{\eff}$ sends the indicated morphism to the identity on
the motive of the common interior.
\end{proof}

\subsection*{Use of AI tools}
The author used ChatGPT 6 Astra (OpenAI) to correct grammatical
errors in the English text, check for typographical errors,
and identify mistakes in the author's calculations.
The author takes full responsibility for the content of this paper.

\bibliographystyle{amsplain}
\bibliography{references}

@article{BCKS17,
  author = {Binda, Federico and Cao, Jin and Kai, Wataru and Sugiyama, Rin},
  title = {Torsion and divisibility for reciprocity sheaves and {$0$}-cycles with modulus},
  journal = {J. Algebra},
  volume = {469},
  year = {2017},
  pages = {437--463},
  doi = {10.1016/j.jalgebra.2016.07.036}
}

@article{KMSY21a,
  author = {Kahn, Bruno and Miyazaki, Hiroyasu and Saito, Shuji and Yamazaki, Takao},
  title = {Motives with modulus, {I}: Modulus sheaves with transfers for non-proper modulus pairs},
  journal = {\'Epijournal G\'eom. Alg\'ebrique},
  volume = {5},
  year = {2021},
  pages = {Paper No. 1, 46},
  doi = {10.46298/epiga.2021.volume5.5979},
  eprint = {1908.02975},
  archivePrefix = {arXiv}
}

@article{KMSY21b,
  author = {Kahn, Bruno and Miyazaki, Hiroyasu and Saito, Shuji and Yamazaki, Takao},
  title = {Motives with modulus, {II}: Modulus sheaves with transfers for proper modulus pairs},
  journal = {\'Epijournal G\'eom. Alg\'ebrique},
  volume = {5},
  year = {2021},
  pages = {Paper No. 2, 31},
  doi = {10.46298/epiga.2021.volume5.5980},
  eprint = {1910.14534},
  archivePrefix = {arXiv}
}

@article{KMSY22,
  author = {Kahn, Bruno and Miyazaki, Hiroyasu and Saito, Shuji and Yamazaki, Takao},
  title = {Motives with modulus, {III}: The categories of motives},
  journal = {Ann. K-Theory},
  volume = {7},
  number = {1},
  year = {2022},
  pages = {119--178},
  doi = {10.2140/akt.2022.7.119},
  note = {arXiv:2011.11859v3}
}

@article{KS21,
  author = {Kelly, Shane and Saito, Shuji},
  title = {Smooth blowup square for motives with modulus},
  journal = {Bull. Polish Acad. Sci. Math.},
  volume = {69},
  number = {2},
  year = {2021},
  pages = {97--106},
  doi = {10.4064/ba190820-10-3},
  eprint = {1907.12759},
  archivePrefix = {arXiv}
}

@article{M19,
  author = {Miyazaki, Hiroyasu},
  title = {Cube invariance of higher {Chow} groups with modulus},
  journal = {J. Algebraic Geom.},
  volume = {28},
  number = {2},
  year = {2019},
  pages = {339--390},
  doi = {10.1090/jag/726},
  eprint = {1604.06155},
  archivePrefix = {arXiv}
}

@article{M23a,
  author = {Matsumoto, Keiho},
  title = {{Gysin} triangles in the category of motifs with modulus},
  journal = {J. Inst. Math. Jussieu},
  volume = {22},
  number = {5},
  year = {2023},
  pages = {2131--2154},
  doi = {10.1017/S1474748021000554},
  note = {arXiv:1812.10890v4}
}

@article{T18,
  author = {Temkin, Michael},
  title = {Functorial desingularization over {$\mathbb{Q}$}: boundaries and the embedded case},
  journal = {Israel J. Math.},
  volume = {224},
  number = {1},
  year = {2018},
  pages = {455--504},
  doi = {10.1007/s11856-018-1656-6}
}

@incollection{V00,
  author = {Voevodsky, Vladimir},
  title = {Triangulated categories of motives over a field},
  booktitle = {Cycles, transfers, and motivic homology theories},
  series = {Annals of Mathematics Studies},
  volume = {143},
  pages = {188--238},
  publisher = {Princeton University Press},
  address = {Princeton, NJ},
  year = {2000}
}

\end{document}